\documentclass[10pt]{article}
\usepackage{epsfig}
\usepackage{amssymb,amsmath,amsthm,url}
\usepackage{multirow}
\usepackage{booktabs}
\usepackage{natbib}
\usepackage{hyperref}
\usepackage{setspace}
\usepackage{dsfont}

\usepackage{times,amsfonts,eucal}
\usepackage{graphicx,ifthen}
\usepackage{wrapfig}
\usepackage{latexsym}
\usepackage{color}
\usepackage{mathrsfs}
\usepackage{bm}
\usepackage{bbm}
\usepackage{srctex}

\usepackage{mathtools}
\DeclarePairedDelimiter{\floor}{\lfloor}{\rfloor}

\theoremstyle{plain}
\newtheorem{theorem}{Theorem}[section]
\newtheorem{corollary}[theorem]{Corollary}

\newtheorem{proposition}[theorem]{Proposition}

\theoremstyle{definition}

\theoremstyle{remark}

\numberwithin{equation}{section}

\usepackage{bbm}

\newcommand{\N}{\mathbbm{N}}
\newcommand{\R}{\mathbbm{R}}

\newcommand{\Z}{\mathbbm{Z}}

\newcommand{\Rplus}{\mathbbm{R}_{\geq 0}}
\newcommand{\p}{\mathbbm{P}}

\newcommand{\cA}{\mathcal{A}}

\newcommand{\cG}{\mathcal{G}}
\newcommand{\cF}{\mathcal{F}}
\newcommand{\cH}{\mathcal{H}}

\newcommand{\E}[1]{\mathbbm{E}\left [ \, #1 \, \right ]}

\renewcommand{\epsilon}{\varepsilon}
\renewcommand{\phi}{\varphi}
\newcommand{\pspace}{(\Omega,\cA,\p)}
\newcommand{\intd}[1]{\,\mathrm{d}#1}
\newcommand{\norm}[1]{\left\lVert #1 \right\rVert}

\allowdisplaybreaks
\begin{document}

\title{A Bernstein Inequality For Spatial Lattice Processes\footnote{This research was partially supported by the Fraunhofer ITWM, 67663 Kaiserslautern, Germany and by the RTG 1932 Stochastic Models for Innovations in the Engineering Sciences funded by the German Research Foundation (DFG).}
}

\author{
Eduardo Valenzuela-Dom{\'i}nguez\footnote{Departamento de Matem{\'a}tica, Universidad T{\'e}cnica Federico Santa Mar{\'i}a, Valpara{\'i}so, Chile, email: \tt{eduardo.valenzuela@usm.cl}}
\and
Johannes T. N. Krebs\footnote{Department of Mathematics, University of Kaiserslautern, 67653 Kaiserslautern, Germany, email: \tt{krebs@mathematik.uni-kl.de} }\;\footnote{Corresponding author}
\and
J{\"u}rgen E. Franke\footnote{Department of Mathematics, University of Kaiserslautern, 67653 Kaiserslautern, Germany, email: \tt{franke@mathematik.uni-kl.de} }
}

\date{\today}
\maketitle
\linespread{1.1}
\begin{abstract}
\setlength{\baselineskip}{1.8em}
%% Text of abstract
In this article we present a Bernstein inequality for sums of random variables which are defined on a spatial lattice structure. The inequality can be used to derive concentration inequalities. It can be useful to obtain consistency properties for nonparametric estimators of conditional expectation functions. \medskip
\end{abstract}

%% main text
\section{Introduction}~\label{Introduction}
Inequalities of the Bernstein type are a major tool for the asymptotic analysis in probability theory and statistics. The original inequality published by \cite{bernstein1927} considers the case $\p( |S_n| \ge \epsilon)$, where $S_n=\sum_{k=1}^n Z_k$ for real-valued zero-mean random variables $Z_1,\ldots,Z_n$ which are independent and identically distributed and bounded. A short proof is given in \cite{bosq2012linear} together with a demonstration how Hoeffding's inequality (\cite{hoeffding1963probability}) can be concluded too. A version for independent multivariate random variables is given by \cite{ahmad2013probability}.

Starting with Collomb's and Carbon's inequalities (\cite{collomb1984proprietes} and \cite{carbon1983}), during the last thirty years there have been derived various generalizations of Bernstein's inequality to stochastic processes $\{ Z(t): t\in\Z\}$ under the assumption of weak dependence (\cite{bryc1996large} and \cite{merlevede2009}). The corresponding definitions of dependence and their interaction properties can be found in \cite{doukhan1994mixing} and in \cite{bradley2005basicMixing}.

Furthermore, there are inequalities of the Bernstein-type which are tailored to special mathematical questions: \cite{arcones1995bernstein} develop Bernstein-type inequalities for $U$-statistics. \cite{krebs2017_BernsteinExpGraph} gives an exponential inequality for strongly mixing random fields which are defined on exponentially growing graphs.

Bernstein inequalities often find their applications when deriving large deviation results or (uniform) asymptotic consistency statements in nonparametric regression and density estimation: \cite{valenzuela1995} considers nonlinear function estimation on random random fields under mixing conditions. Such statistical procedures are also widely used in image analysis, where the image is modeled as a given function on part of the integer lattice $\Z^2$ contaminated by additive noise. Frequently, the noise is assumed to consist of independent and identically distributed random variables, but this assumption is not always realistic, compare e.g., \cite{daul1998fast}. A more general noise model is provided by stationary stochastic processes, e.g., by Markov random fields. For such processes, functions like conditional probability densities or conditional expectations of an observation given data in a neighborhood may also be estimated by nonparametric procedures \cite{tran1990kernel}. For investigating the asymptotic properties of those estimation procedures a Bernstein inequality for spatial stochastic processes on an integer lattice is needed. For continuous-parameter processes on $\R^2$, such a result has been derived in \cite{bertail2000subsampling}. Here, we provide a Bernstein inequality for stochastic processes on $\Z^N$ under rather general conditions, e.g., assuming only $\alpha$-mixing which is a rather weak type of mixing condition. To allow for other applications, e.g., to spatial-temporal processes
used in modeling environmental data like precipitation or pollution, we do not restrict ourselves to the plane but consider integer lattices in arbitrary dimensions.

This paper is organized as follows: we give the main definitions and notation in Section~\ref{Section_DefinitionsNotation}. In Section~\ref{Section_ExponentialInequalities} we present the Bernstein inequality for random fields on a lattice $\Z^N$ and further concentration inequalities, it is the main part of this article.

\section{Definitions and Notation}\label{Section_DefinitionsNotation}

In this section we give the mathematical definitions and notation which we shall use to derive the results. We work on a probability space $\pspace$. Let $N\in\N$ be a natural number. A real-valued random field $Z$ which is indexed by $\Z^N$ is a collection of random variables $\{Z(s): s\in \Z^N\}$. We write $d_{\infty}$ for the metric on the lattice $\Z^N$ which is induced by the Euclidean-$\infty$-norm, i.e., $d_{\infty}(s,t) = \max\{ |s_i - t_i| : 1\le i \le N \}$ for $s,t\in \Z^N$. Denote for two subsets $I, J\subseteq \Z^N$ their distance by
$$ d_{\infty}(I,J) = \inf\{ d_{\infty}(s,t): s\in I, t\in J \}. $$
Furthermore, we write $s\le t$ if and only if $s_i \le t_i$ for $i=1,\ldots,N$.

The $\alpha$-mixing coefficient is introduced by \cite{rosenblatt1956central}. It is defined by
\begin{align*}
		\alpha(\cF,\cG) \coloneqq \sup \left\{  \left| \p(A\cap B)-\p(A)\p(B)		\right|: A\in\cF, B\in\cG \right\}
\end{align*}
for two sub-$\sigma$-algebras $\cF$ and $\cG$ of $\cA$. Note that $\alpha(\cF,\cG) \le 1/4$, compare \cite{bradley2005basicMixing}. If $X$ and $Y$ are two random variables on $\pspace$, then $\alpha(X,Y)$ is the mixing coefficient $\alpha(\sigma(X), \sigma(Y) )$. Furthermore, for a random field $\{Z(s): s\in \Z^N\}$ and a subset $I\subseteq\Z^N$, denote by $\cF(I) \coloneqq \sigma( Z(s): s\in I )$ the $\sigma$-algebra generated by the $Z(s)$ in $I$. The $\alpha$-mixing coefficient of the random field $Z$ is then defined as
\begin{align}\label{StrongSpatialMixing}
	\alpha(k) \coloneqq \sup_{ \substack{I, J \subseteq \Z^N,\\ d_{\infty}(I,J)\ge k }} \sup_{ \substack{ A \in \cF(I),\\ B \in \cF(J) } } \left| \p(A\cap B)-\p(A)\p(B)		\right|, \quad k\in\N.
\end{align}
The random field $Z$ is said to be strongly (spatial) mixing (or $\alpha$-mixing) if $\alpha(k)\rightarrow 0$ ($k\rightarrow \infty$).

We write $e_N = (1,\ldots,1)$ for the element in $\Z^N$ which only contains ones. Let $n=(n_1,\ldots,n_N)\in \N^N$, then we write $I_n$ for the $N$-dimensional cube on the lattice which is spanned by $e_N$ and $n$, i.e., $I_n = \{ k\in \Z^N: e_N\le k \le n \}$.

\section{Exponential inequalities for \texorpdfstring{$\alpha$}{alpha}-mixing processes on \texorpdfstring{$N$}{N}-dimensional lattices}\label{Section_ExponentialInequalities}

\begin{theorem}[Bernstein inequality]\label{BernsteinLattice}
Let $Z:=\{ Z(s): s \in \Z^N \}$ be a real-valued random field defined on the $N$-dimensional lattice $\Z^N$. Each $Z(s)$ is bounded by a uniform constant $B$, has expectation zero and the variance of $Z(s)$ is uniformly bounded by $\sigma^2$. Let $Z$ be strongly mixing with mixing coefficients $\{\alpha(k): k\in \N\}$. Set $\bar{\alpha}_k := \sum_{u=1}^k u^{N-1} \alpha(u)$. Let $P(n) = (P_1(n_1),\ldots,P_N(n_N) )$ and $Q(n)=(Q_1(n_1),\ldots,Q_N(n_N) )$ be arbitrary non-decreasing sequences in $\N^N$ which are indexed by $n\in\N^N$ and which satisfy for each $1\le k \le N$
\begin{align}
		&1 \le Q_k(n_k) \le P_k(n_k) < Q_k (n_k) + P_k(n_k) < n_k. \label{EqBernsteinLattice0a}
\end{align}
Furthermore, let ${ \bf n } := |I_n| = n_1\cdot \ldots \cdot n_N$, ${ \bf P } := P_1(n_1) \cdot \ldots \cdot P_N (n_N)$ and $\underline{q} := \min\left\{ Q_1(n_1),\ldots,Q_N(n_N) \right\}$ as well as $\overline{p} := \max\left\{ P_1(n_1), \ldots, P_N(n_N) \right\}$. Then for all $\epsilon > 0$ and $\beta >0$ such that $2^{N+1} B { \bf P } e \beta < 1$
\begin{align}\begin{split}\label{EqBernsteinLattice0b}
		\p\left( \left|\sum_{s \in I_n} Z(s) \right| \ge \epsilon \right) &\le 2 \exp\left\{ 12 \sqrt{e} 2^N \frac{{ \bf n }}{{ \bf P }} \alpha(\underline{q})^{ { \bf P } \big/ \left[ { \bf n } \left(2^N+1\right) \right] }	\right\} \\
		&\qquad\qquad\qquad \cdot  \exp\Big\{	-\beta \epsilon + 2^{3N} \beta^2 e \left(	\sigma^2 + 12  B^2\, \gamma \bar{\alpha}_{\overline{p}}\right) { \bf n }	\Big\},
\end{split}\end{align}
where the constant $0<\gamma<\infty $ depends on the lattice dimension $N$.
\end{theorem}
\begin{proof}
We write $S_n = \sum_{s\in I_n} Z(s)$ for $n\in \N^N$. To exploit the mixing property we want to decompose the sum $S_{n}$ into
different parts which consist of sums over groups of the $Z(s).$ Using
the mixing condition, most of these subsums are only weakly dependent. To simplify notation, we write
$$P \equiv P(n) \equiv (P_1, \ldots, P_N),\quad Q
\equiv Q(n) \equiv (Q_1, \ldots, Q_N)$$
keeping the dependence on ${n}$ in mind. We choose a corresponding
sequence $R \equiv R(n) \equiv (R_1, \ldots, R_N)$ such that
\begin{equation}\label{EqBernsteinLattice1}
(R_k-1)(P_k+Q_k) < n_k \le R_k(P_k + Q_k) =: n_k^* \text{ for each } k=1, \ldots, N.
\end{equation}
For the $k$-th coordinate direction, we partition the summation index set
$\{1, \ldots, n_k^*\} \supseteq \{ 1, \ldots, n_k\}$ into $R_k$ subsets each
consisting of two disjoint intervals of length $P_k$ and $Q_k$ resp. So, we
have a union of $2 R_k$ intervals half of them of length $P_k$, the other half
of length $Q_k,$ covering the set $\{1, \ldots, n_k\}.$ 

Combining the partitions in all $N$ coordinate directions, we get a partition of the
$N$-dimensional rectangle $I_{n^*} = \{ s \in \Z^N; e_N \le s \le n^* \} \supseteq I_{n}$ into ${ \bf R } = R_1 \cdot \ldots \cdot R_N$ blocks containing $(P_1 + Q_1) \cdot \ldots \cdot (P_N+Q_N)$ points of the $N$-dimensional integer lattice each. Within each block, there are $2^N$ smaller subsets, which are $N$-dimensional rectangles with all edges of length either $P_k$ or $Q_k,\ k=1, \ldots, N$. Write $I(l,u)$ for the $l$-th subset in the $u$-th block, $l=1,\ldots,2^N$ and $u=1,\ldots,{ \bf R }$. Note that the diameter w.r.t.\ $d_{\infty}$ of the rectangular set $I(l,u)$ is bounded by $\overline{p}$, since
\begin{align}\label{EqBernsteinLattice2}
		\operatorname{diam} \{I(l,u)\} = \max\{d_\infty (s,t), s,t \in I(l,u)\} \le \max \{ P_1,\ldots,P_N \} = \overline{p}.
\end{align}
Its cardinality is at most $\operatorname{card} \{I(l,u)\} \le \prod_{k=1}^N \max \{ P_k, Q_k\} = \prod_{k=1}^N P_k = { \bf P }$ (cf. \eqref{EqBernsteinLattice0a}. Now we can partition the sum $S_{n} = \sum_{s \in I_n} Z(s)$ as follows
\begin{equation*}
S_{n} = \sum_{l=1}^{2^N} \sum_{u=1}^{{ \bf R }} \sum_{ s \in I(l,u)}
            Z(s) =  \sum_{l=1}^{2^N} \sum_{u=1}^{{ \bf R }} S(l,u)
          = \sum_{l=1}^{2^N} T(l,{ \bf R })
\end{equation*}
with $S(l,u) = \sum_{ s \in I(l,u)} Z(s)$ and $T(l,r) = \sum_{u=1}^r S(l,u)$, for $r=1,\ldots,{ \bf R }$. We have the recursive property
\begin{equation}\label{EqBernsteinLattice3}
T(l,r) = T(l,r-1) + S(l,r) \text{ and } T(l,0) = 0.
\end{equation}
Now we can apply this decomposition to the exponential $e^{\beta S_{n}}$ as follows
\begin{eqnarray} \label{EqBernsteinLattice4}
\E{ e^{\beta S_{n}} }&=  \E{ e^{\beta \sum_{l=1}^{2^N}  T(l,{ \bf R })} } = \E{ \prod_{l=1}^{2^N} e^{\beta T(l,{ \bf R })} } \le \E{ 2^{-N} \sum_{l=1}^{2^N} e^{2^N \beta  T(l,{ \bf R })}} 
\end{eqnarray}
where we have used the well-known inequality between geometric and arithmetic mean. Setting $\delta = 2^N \beta$ we have $\E{ e^{\beta S_{n}} } \leq 2^{-N} \sum_{l=1}^{2^N} \E{ e^{\delta T(l,{ \bf R })} }$. Now, we  study $\E{ e^{\delta T(l,r)} }$ for $l=1,\ldots,2^N$ and
$r=1,\ldots,{ \bf R }.$ By \eqref{EqBernsteinLattice3}
\begin{eqnarray*}
\E{ e^{\delta T(l,r)} } & = & \E{ e^{\delta T(l,r-1)}e^{\delta S(l,r)} }\\
                      & \leq & \left|\E{ e^{\delta T(l,r-1)}e^{\delta S(l,r)} } -
                               \E{e^{\delta T(l,r-1)} } \E{ e^{\delta S(l,r)} } \right| +
                              \left |\E{ e^{\delta T(l,r-1)} } \E { e^{\delta S(l,r)} } \right|.
\end{eqnarray*}
But $T(l,r-1)$ is ${\cal F}(I(l,1)\cup \cdots \cup I(l,r-1)) =:
{\cal F}(J(l,r-1))$-measurable and $S(l,r)$ is ${\cal F}(I(l,r))$-measurable, this implies that $e^{\delta T(l,r-1)}$ is ${\cal F}(J(l,r-1))$-measurable and $e^{\delta S(l,r)}$ is ${\cal F}(I(l,r))$-measurable. Since $Z(s)$ is bounded and the minimal distance between the sets $J(l,r-1)$ and $I(l,r)$ is $d_{\infty}(J(l,r-1),I(l,r)) \geq \min \{ Q_1,\ldots,Q_N \} = \underline{q} ,$
we can apply Davydov's inequality (compare \ref{davydovsIneq}) as follows
\begin{equation*}
\left| \E{ e^{\delta T(l,r-1)}e^{\delta S(l,r)}}  - \E{ e^{\delta T(l,r-1)}} \E{ e^{\delta S(l,r)} } \right|
\leq
12 \alpha(\underline{q})^{1/a}
\|e^{\delta S(l,r)}\|_{\infty}
\|e^{\delta T(l,r-1)}\|_b
\end{equation*}
with $a,b \geq 1$ such that $\frac{1}{a} + \frac{1}{b} = 1$, therefore
\begin{equation}\label{EqBernsteinLattice5}
	\E{ e^{\delta T(l,r)} }
\leq
12\alpha(\underline{q})^{1/a}
\|e^{\delta S(l,r)}\|_{\infty}
\|e^{\delta T(l,r-1)}\|_b
+
	\E{ e^{\delta T(l,r-1)} } \E{e^{\delta S(l,r)} }.
\end{equation}
As $|S(l,r)| \leq \sum_{ s \in I(l,r)} |Z(s)| \le B { \bf P } $ and choosing
\[
0 < \beta \leq \frac{1}{2^{N+1} B { \bf P } e},\ \text{i.e.,\ } 0 < \delta \le \frac{1}{2B{ \bf P } e}
\]
we have $\delta S(l,r) | \leq 1/(2e)$ and for all $D$ such that $0 \leq D \leq e$
\begin{equation}\label{EqBernsteinLattice6}
|\delta D S(l,r)| \leq \frac{1}{2}
\end{equation}
which implies
\begin{eqnarray}\label{EqBernsteinLattice7}
	\norm{ e^{\delta D S(l,r)} }_{\infty} & \leq & \sqrt{e}.
\end{eqnarray}
Using \eqref{EqBernsteinLattice6}, we have $e^{\delta D S(l,r)}  \le  1 + \delta D S(l,r) + (\delta D S(l,r))^2$. Next, we take expectations of this inequality and use that the $Z(s)$ have expectation zero as well as that the inequality $1 + x \le \exp x$ is true for all $x\ge 0$. We obtain
\begin{align} \label{EqBernsteinLattice8}
	\E{ e^{\delta D S(l,r)} } \le 1 + {\delta}^2 D^2 \E{ S(l,r)^2 } \le e^{ {\delta}^2 D^2 \E{ S(l,r)^2} }.
\end{align}
Now we have to evaluate $E[S(l,r)^2]$:
\begin{align*}
\E{ S(l,r)^2 }  = \E{ \left( \sum_{ s \in I(l,r)} Z(s) \right)^2 } = \sum_{ s \in I(l,r)} \E{ Z(s) ^2 } + \sum_{s \in I(l,r)} \sum_{{t \in I(l,r), t \not= s}} \E{Z(s) Z(t)} 
\end{align*}
We know that $|Z(s)| \leq B$, so $|\E{ Z(s) Z(t) }| \le 12 B^2 \alpha(d_{\infty}(s,t) )$ and using $\E{ Z(s)^2 } \le \sigma^2<\infty,$  we have
\[
	\E{ S(l,r)^2 } \leq \sigma^2 { \bf P } +
12 B^2 \sum_{s \in I(l,r)} \sum_{{t \in I(l,r), t \not= s}}
\alpha(d_{\infty}(s,t))
\]
In order to evaluate the double sum,  note that if $s,t \in I(l,r), s \not= t$, then by \eqref{EqBernsteinLattice2} $d_{\infty}(s,t)$ assumes values between 1 and $\overline{p}$, i.e., $1 \leq d_{\infty}(s,t) \le \overline{p}$. Furthermore, for a general point $s\in\Z^N$ the cardinality of the set of points $t\in\Z^N$ whose distance to $s$ is exactly $u$ is $\operatorname{card}\{t \in \Z^N : d_{\infty}(s,t) = u \} = (2u+1)^N - (2u-1)^N \le \gamma u^{N-1}$ for $u\ge 1$, where $\gamma$ is a constant which depends on the lattice dimension $N$. Thus, the double sum can be bounded as follows
\begin{align*}
	\sum_{s\in I(l,r)} \sum_{t \in I(l,r), t \not= s} \alpha(d_{\infty}(s,t)) & \le \sum_{s\in I(l,r)} \sum_{u=1}^{\overline{p}} \sum_{t\in\Z^N: d_{\infty}(s,t) = u } \alpha(u) \\
	&\le \sum_{s\in I(l,r)} \sum_{u=1}^{\overline{p}}  \alpha(u) \left\{ (2u+1)^N - (2u-1)^N \right\} \le \gamma\, { \bf P } \sum_{u=1}^{\overline{p}}  \alpha(u) u^{N-1}.
\end{align*}
So, we have
\begin{equation}\label{EqBernsteinLattice9}
E[S(l,r)^2] \leq  \sigma^2 { \bf P } + 12 B^2 \gamma \bar{\alpha}_p { \bf P }.
\end{equation}
From \eqref{EqBernsteinLattice8} we obtain $\E{ e^{\delta D S(l,r)} } \le \exp(\delta^2 D^2 (\sigma^2 { \bf P } + 12 B^2 \gamma \bar{\alpha}_p { \bf P })$. We set $V \coloneqq \sigma^2 { \bf P } + 12 B^2 \gamma \bar{\alpha}_p { \bf P }$ and $D=1$. Thus, it follows from \eqref{EqBernsteinLattice5}
\begin{equation*}
	\E{ e^{\delta T(l,r)} } \leq 12 \alpha(\underline{q})^{1/a} \norm{ e^{\delta S(l,r)}}_{\infty} \norm{ e^{\delta T(l,r-1)}}_b + \E{ e^{\delta T(l,r-1)} } e^{\delta^2 V}.
\end{equation*}
But by H\"older's inequality  $\E{ e^{\delta T(l,r-1)} } \leq \norm{ e^{\delta T(l,r-1)}}_b$, so we obtain
\begin{equation}\label{EqBernsteinLattice10} 
	\E{ e^{\delta T(l,r)} } \le \left( e^{\delta^2 V} + 12 \alpha(\underline{q})^{1/a}
\norm{ e^{\delta S(l,r)}}_{\infty} \right) \norm{ e^{\delta T(l,r-1)} }_b.
\end{equation}
Now let $a= 1 + r$ and $b = 1 + 1/r$ such that for all $i=1,\ldots,r$, we have
\begin{equation}\label{EqBernsteinLattice11}
1 \leq b^{i-1} \leq \left ( 1 + \frac{1}{r} \right )^r \leq e.
\end{equation}
Then we obtain successively as in deriving \eqref{EqBernsteinLattice10} the following inequalities for $r \geq 2$:
\begin{eqnarray*}
\norm{e^{\delta T(l,r-1)}}_b
& \leq &
\left( e^{\delta^2 b^2 V} +
12 \alpha(\underline{q})^{1/a} \norm{e^{\delta b S(l,r-1)}}_{\infty} \right)^{1/b}
\norm{e^{\delta T(l,r-2)}}_{b^2}\\
\norm{e^{\delta T(l,r-2)}}_{b^2}
& \leq &
\left( e^{\delta^2 b^4 V} +
12 \alpha(\underline{q})^{1/a} \norm{e^{\delta b^2 S(l,r-2)}}_{\infty} \right)^{1/b^2}
\norm{e^{\delta T(l,r-3)}}_{b^3}\\
& \vdots & \\
\norm{e^{\delta T(l,2)}}_{b^{r-2}}
& \leq &
\left( e^{\delta^2 b^{2(r-2)} V} +
12 \alpha(\underline{q})^{1/a} \norm{e^{\delta b^{r-2} S(l,2)}}_{\infty} \right)^{ 1 / b^{r-2}  }
\norm{e^{\delta T(l,1)}}_{b^{r-1}}.
\end{eqnarray*}
Substituting, we get:
\begin{equation}\label{EqBernsteinLattice12}
 \E{ e^{\delta T(l,r)} } \le \left [ \prod_{i=1}^{r-1}
( e^{ \delta^2 b^{2(i-1)}V} + 12\alpha(\underline{q})^{1/a}
\|e^{\delta b^{i-1} S(l,r-i+1)} \|_{\infty} )^{1/b^{i-1}} \right ]
 \E{ e^{\delta b^{r-1} T(l,1)} }^{1/b^{r-1}}
\end{equation}
but $b^{i-1} \leq e$ for $i=1,\ldots,r$, such that $\|e^{\delta b^{i-1} S(l,r-i+1)}\|_{\infty} \leq \sqrt{e}$ by \eqref{EqBernsteinLattice7} and even further
\begin{eqnarray*}
( e^{ \delta^2 b^{ 2(i-1) } V } +
12 \alpha(\underline{q})^{ 1/a } \sqrt{e} )
^{ 1/ b^{i-1} }
& \leq &
e^{ \delta^2 b^{i-1} V } ( 1 + 12 \alpha(\underline{q})^{ 1/a }
\sqrt{e} )^{ 1/b^{i-1} }\\
& \leq &
\exp \left\{ \delta^2 b^{i-1} V  + \frac{ 12 \sqrt{e} \alpha(\underline{q})
^{ 1/a } }{ b^{i-1} } \right \}\\
& \leq &
\exp \left\{ \delta^2 b^{i-1} V \right\} \exp\left\{ 12 \sqrt{e} \alpha(\underline{q})^{ 1/a } \right\}
\end{eqnarray*}
by \eqref{EqBernsteinLattice11}. Therefore, again using \eqref{EqBernsteinLattice11}
\begin{align*}
\prod_{i=1}^{r-1}
(
e^{ \delta^2 b^{ 2(i-1) } V }
+
12 \alpha(\underline{q})^{ 1/a } \sqrt{e}
)
^{
1/ b^{i-1}
 }
 \leq 
\prod_{ i=1 }^{r-1}
e^{12 \sqrt{e}\alpha(\underline{q})^{1/a}
+
\delta^2 e V} =
\exp \left\{12 \sqrt{e}\alpha(\underline{q})^{1/a} (r-1)
+
\delta^2 e V (r-1) \right\}.
\end{align*}
Since $b^{r-1} \leq e$ by \eqref{EqBernsteinLattice11}, and using \eqref{EqBernsteinLattice8} and
\eqref{EqBernsteinLattice9} we have
\begin{align*}
\|e^{\delta T_{l,1}}\|_{b^{r-1}} &\le \|e^{\delta T(l,1)}\|_e = \E{ e^{\delta e T(l,1)}}^{ 1/e } \le \E{e^{\delta eS(l,1)}}^{ 1/e }\\
          & \le (e^{\delta^2 e^2 \E{S(l,1)^2}})^{ 1 /e } \le (e^{\delta^2 e^2 V})^{1/e } =  \exp\{ \delta^2 e V \}.
\end{align*}
Combining these results, we get from \eqref{EqBernsteinLattice12}  for $l=1,\ldots,2^N$ and $r=1,\ldots,{ \bf R }$ that
\begin{equation*}
	\E{ e^{\delta T(l,r)} } \le \exp\left\{12 \sqrt{e}\alpha(\underline{q})^{1/a}(r-1) + \delta^2 e V r\right\}.
\end{equation*}
By \eqref{EqBernsteinLattice0a}, $P_k < P_k + Q_k < n_k$ for each $k=1,\ldots,N$ which implies by \eqref{EqBernsteinLattice1} that both $
P_k < n_k $  and $R_k < 2n_k /P_k $. For $a=1+r, r={ \bf R },$ we therefore have the two relations $1 > \frac{1}{a} > { \bf P } / [ (2^N+1){ \bf n } ]$ and ${ \bf R }\le 2^N { \bf n } / { \bf P }$. Hence, for the choice $r= { \bf R }$ we arrive at (using that $0 < \alpha(\underline{q}) \leq 1/4$)
\begin{eqnarray*}
\E{ e^{\delta T(l,{ \bf R })} }
& \leq &
\exp\left\{  12\sqrt{e}\alpha(\underline{q})^{ { \bf P } / [ (2^N+1){ \bf n }] }
\left(2^N\frac{{ \bf n }}{{ \bf P }}-1\right)
+
\delta^2eV 2^N \frac{{ \bf n }}{{ \bf P }} \right\}.
\end{eqnarray*}
Using $\delta = 2^N \beta$:
\begin{equation*}
\E{ e^{2^N \beta T(l,{ \bf R })} } \le \exp\left \{ 2^{3N} \beta^2 e V  \frac{{ \bf n }}{{ \bf P }}+12 \sqrt{e} \alpha(\underline{q})^{ { \bf P } / [(2^N+1) { \bf n } ]} \left(2^N \frac{{ \bf n }}{{ \bf P }}-1 \right) \right\}
\end{equation*}
Returning to \eqref{EqBernsteinLattice4} and using Markov's inequality, we have
\begin{align*}
\p\left( |S_{n}| \ge \varepsilon \right) &= \p\left( S_{n} \ge \varepsilon \right) + \p\left(-S_{n} \ge \varepsilon \right)  = \p\left( e^{\beta S_{n}} \ge e^{\beta \varepsilon} \right) + \p\left( e^{-\beta S_{n}} \ge e^{\beta \varepsilon} \right) \\
& \le e^{-\beta \epsilon} \left\{ \E{ e^{\beta S_{n}}} + \E{ e^{-\beta S_{n}} } \right\} 
\end{align*}
Now, if we change $Z(s)$ to $-Z(s)$, all results remain valid, therefore we have in \eqref{EqBernsteinLattice4}
\begin{eqnarray*}
\p\left( |S_{n}| \ge \varepsilon \right)
& \leq &
e^{-\beta \varepsilon}
\left \{
2^{-N}
\sum_{l=1}^{2^N}
\left(
 \E{ e^{2^N \beta T_{l,{ \bf R } } } }
+
\E{e^{-2^N \beta T_{l,{ \bf R } } } }
\right)
\right \}\\
& \leq &
2 e^{-\beta \varepsilon}
\exp
\left \{
2^{3N} \beta^2 e V \frac{{ \bf n }}{{ \bf P }}
+
12\sqrt{e}\alpha(\underline{q})^{ { \bf P } / [ (2^N+1){ \bf n } ] }
\left ( 2^N \frac{ { \bf n } }{ { \bf P } } - 1 \right )
\right \}.
\end{eqnarray*}
Recalling the definition of  $V$ this immediately implies \eqref{EqBernsteinLattice0b}.
\end{proof}

We can formulate the following extension of the above Bernstein inequality for unbounded random variables
\begin{theorem}\label{extBernstein}
Let $\{Z(s): s\in I \}$ be a strongly mixing random field with $\E{Z(s)} = 0$ and $\E{Z(s)^2} \le \sigma^2 < \infty$. Furthermore, assume that the tail distribution is bounded uniformly in $s$ by
\begin{align}
		\p(|Z(s)| \ge z ) \le \kappa_0 \exp\left( - \kappa_1 z^{\tau} \right) \label{tailBound}
\end{align}
for $\kappa_0,\kappa_1, \tau >0$. Then for any $B>0$ it is true that
\begin{align*}
		\p\left( \left| \sum_{s\in I_n} Z(s) \right| \ge \epsilon \right) &\le \frac{12}{\epsilon \tau} \kappa_0  \kappa_1^{-1/ \tau } \Gamma\left( \tau^{-1}, \kappa_1 B^{\tau} \right) { \bf n } + 2 \exp\left\{ 12 \sqrt{e}2^N \frac{{ \bf n }}{{ \bf P }} \alpha(\underline{q})^{ { \bf P }\big/\left[{ \bf n } \left(2^N+1\right) \right] } \right\} \\
		&\qquad\qquad\qquad\qquad \cdot \exp\left\{ - \frac{1}{3} \beta \epsilon \right\} \cdot \exp\Big\{	2^{3N}\beta^2 e\left( \sigma^2 + 48 B^2 \gamma \bar{\alpha}_{\overline{p}} \right) { \bf n }	\Big\},
\end{align*}
where $\Gamma$ denotes the upper incomplete gamma function.
\end{theorem}
\begin{proof}%[Proof of Theorem \ref{extBernstein}]
We split each $Z(s)$: choose an arbitrary bound $B>0$ and define for $s\in\Z^N$
\begin{align*}
		Z(s)^{\#} \coloneqq Z(s) - \min( Z(s), B) \ge 0, \quad Z(s)^{*} &\coloneqq Z(s) - \max(Z(s), -B) \le 0 \\
		&\text{ and } Z(s)^0 \coloneqq \max( \min(Z(s), B), -B).
\end{align*}
Then, $Z(s) = Z(s)^{\#} + Z(s)^{*} + Z(s)^0$ and $0 = \E{Z(s)} = \E{Z(s)^{\#}} + \E{Z(s)^{*}} + \E{Z(s)^0}$. Thus,
\begin{align}
	&\p\left( \left| \sum_{s\in I_n} Z(s) \right|  \ge \epsilon		\right) = \p\left( \left|\sum_{s \in |I_n | } Z(s) - \E{Z(s)} \right| \ge \epsilon \right) \nonumber \\
	\begin{split}
	&\le \p\left( \left| \sum_{s \in |I_n | } Z(s)^{\#} - \E{Z(s)^{\#}} \right| \ge \frac{\epsilon}{3} \right) + \p\left( \left| \sum_{s \in |I_n |  } Z(s)^{*} - \E{Z(s)^{*}} \right| \ge \frac{\epsilon}{3} \right) \\
	&\quad + \p\left( \left| \sum_{s\in |I_n |  } Z(s)^0 - \E{Z(s)^0} \right | \ge \frac{\epsilon}{3} \right). \label{extUSLLNM1}
\end{split}\end{align}
We treat each term in \eqref{extUSLLNM1} separately. We consider the first two terms. We obtain with Markov's inequality
\begin{align}
		\p\left( \left| \sum_{s \in |I_n | } Z(s)^{\#} - \E{Z(s)^{\#}} \right| \ge \frac{\epsilon}{3} \right) \le \frac{3}{\epsilon} \E{ \left| \sum_{s \in |I_n | } Z(s)^{\#} - \E{Z(s)^{\#}} \right| } \le  \frac{6 |I_n| }{\epsilon} \E{ Z(s)^{\#} }. \label{extUSLLNM2}
\end{align}
Using the tail condition, we can estimate the expectation in \eqref{extUSLLNM2} by
\begin{align}
		&\E{Z(s)^{\#}} = \int_0^{\infty} \p\left( Z(s)^{\#} \ge z \right) \intd{z} \nonumber \\
		&= \int_0^{\infty} \p\left( (Z(s) - B) 1_{\{Z(s) \ge B\} } \ge z \right) \intd{z} = \int_B^{\infty} \p\left( Z(s)  \ge z \right) \intd{z} \nonumber\\
		&\le \kappa_0 \int_B^{\infty} \exp\left( -\kappa_1 z^{\tau} \right) \intd{z} = \kappa_0 \int_{\kappa_1 B^{\tau}}^{\infty} \frac{1}{\tau} \left(\frac{1}{\kappa_1}\right)^{ 1 / \tau } y^{\frac{1}{\tau}-1} e^{-y} \intd{y} = \frac{\kappa_0}{\tau} \left(\frac{1}{\kappa_1}\right)^{ 1/\tau } \Gamma\left( \frac{1}{\tau}, \kappa_1 B^{\tau} \right). \nonumber
\end{align}
Since $\sigma(Z(s)^0: s\in I) \subseteq \sigma(Z(s): s\in I)$ for any $I \subseteq \Z^N$, the mixing coefficient of the field $\{Z(s)^0: s\in \Z^N\}$ can be estimated by those of $\{Z(s): s\in \Z^N\}$. Furthermore, $\operatorname{Var}(Z(s)^0) \le \sigma^2$ and we can apply Theorem \ref{BernsteinLattice} to the third term of \eqref{extUSLLNM1}, using that $\left|Z(s)^0 - \E{Z(s)^0} \right | \le 2B$. Hence,
\begin{align*}
	&\p\left( \left| \sum_{s \in |I_n |  } Z(s)^0 - \E{Z(s)^0} \right | \ge \frac{\epsilon}{3} \right) \\
	&\le 2 \exp\left\{ 12 \sqrt{e}2^N \frac{{ \bf n }}{{ \bf P }} \alpha(\underline{q})^{  { \bf P } \big/ \left[ { \bf n }(2^N+1) \right] } \right\} \cdot \exp\left\{ - \frac{\epsilon}{3} \beta \right\} \cdot \exp\left\{	2^{3N}\beta^2 e\left(\sigma^2 + 48  B^2 \gamma  \bar{\alpha}_p\right) { \bf n }	\right\}.
	\end{align*}
This finishes the proof.
\end{proof}

We give a result which is an immediate consequence of Theorem \ref{BernsteinLattice}: 
\begin{corollary}\label{applBernstein}
Let the real valued random field $Z$ have $\alpha$-mixing coefficients which are exponentially decreasing, i.e., there are $c_0, c_1 \in \R_+$ such that $\alpha(k ) \le c_0 \exp( - c_1 k)$. The $Z(s)$ have expectation zero and are bounded by $B$. Moreover, $\E{Z(s)^2}\le \sigma^2$. Let $n\in \N^N$ be such that both
$$\min_{1\le i \le N} n_i \ge e^2 \text{ and } \frac{ \min\{ n_i: i=1,\ldots,N \}}{ \max\{ n_i: i=1,\ldots,N \} } \ge C',$$
for a constant $C'> 0$. Then there are constants $A_1, A_2\in \R_+$ which depend on the lattice dimension $N$, the constant $C'$ and the bound on the mixing coefficients but not on $n\in\N^N$ and not on $B$ such that for all $\epsilon> 0$ 
\begin{align*}
		\p\left( \left| \sum_{s\in I_n} Z(s) \right| \ge \epsilon \right) \le A_1 \exp\left( - \frac{A_2 \epsilon^2}{   (\sigma^2 + B^2){\bf n} +  B\epsilon \,   {\bf n} ^{N/(N+1)} \left(\prod_{i=1}^N \log n_i \right) } \right).
\end{align*}
\end{corollary}
\begin{proof}[Proof of Corollary~\ref{applBernstein}]
Define $P_i(n_i) := Q_i(n_i) := \floor*{ n_i^{N/(N+1)}  \, \log n_i }$ for $i=1,\ldots,N$. Furthermore, we denote the smallest coordinate of $n\in \N^N$ by $\underline{n} := \min_{1\le i \le N} n_i$ and the largest coordinate by $\overline{n} = \max_{1\le i \le N} n_i$. Note that $\overline{n} \rightarrow \infty $ implies that $\underline{n} \rightarrow \infty$. We consider the first factor on the RHS of \eqref{EqBernsteinLattice0b} and show that under the stated conditions
\begin{align}\label{convTildeN0}
	\sup \left\{ \exp\left( 12 \sqrt{e} 2^N \frac{{ \bf n }}{{ \bf P }} \alpha(\underline{q})^{ { \bf P }\big/ \left[{ \bf n}(2^N+1) \right] } 		\right): n\in \Z^N, \underline{n}\ge e^2 \right\} < \infty.
\end{align}
By assumption we have that $\alpha(\underline{q}) \le c_1 \exp( -c_2 \underline{q} )$, for two constants $c_1, c_2 \in \Rplus$ and $\underline{q} = \min_{1\le i \le N} Q_i$. Therefore it suffices to show that
\begin{align} \label{convTildeN1}
		\log ( { \bf n } / { \bf P } ) - c_2 /( 2^N+1)\, \underline{q}\, { \bf P } / { \bf n } \rightarrow - \infty \text{ as } \underline{n} \rightarrow \infty.
\end{align}
Note that for $a,b \ge 2$, we have $ab \ge a+b$. Thus, for $\prod_{i=1}^N \log n_i \ge \sum_{i=1}^N \log n_i$ if $\underline{n}$ is at least $e^2$. We make the definition $\eta := N/(N+1)$. Let $\underline{n} \ge e^2$, then for any constant $c\in \R_+$
\begin{align}
		&\log\left(	\left( \prod_{i=1}^N n_i \right)^{1-\eta} \left(\prod_{i=1}^{N} \log n_i \right)^{-1}	\right) - c (\underline{n})^{ \eta} \log \underline{n} \left(	\prod_{i=1}^N n_i \right)^{\eta-1} \left(\prod_{i=1}^N \log n_i \right) \nonumber \\
		&\le (N+1)^{-1} \sum_{i=1}^N \log n_i - c  \frac{(\underline{n})^{ \eta + (\eta-1)} }{ (\overline{n})^{ (N-1)(1-\eta)}} \left(	\log \underline{n} \prod_{i=1}^N \log n_i	\right) \nonumber \\
		&\le (N+1)^{-1} \prod_{i=1}^N \log n_i -c \left(\frac{\underline{n}}{\overline{n}} \right)^{(N-1)/(N+1)} \left(	\log \underline{n} \prod_{i=1}^N \log n_i	\right) \nonumber \\ 
		&= \left(	(N+1)^{-1} - c \left(\frac{\underline{n}}{\overline{n}} \right)^{(N-1)/(N+1)} \, \log \underline{n}	\right) \prod_{i=1}^N \log n_i \rightarrow - \infty \text{ as } \underline{n} \rightarrow \infty. \nonumber
\end{align}
This proves \eqref{convTildeN1} and consequently, that \eqref{convTildeN0} is finite. Thus, we arrive at
\begin{align*}
			\p\left( \left| \sum_{s\in I_n} Z(s) \right| \ge \epsilon \right) \le A_1 \exp\left(- \beta \epsilon + A_2 \beta^2 (\sigma^2 + B^2 ) {\bf n}  		\right) \le A_1 \exp\left( - \beta \epsilon + \frac{1}{2} \frac{2 A_2  (\sigma^2 + B^2 ) {\bf n} \beta^2 }{ 1- 2^{N+1} B {\bf P}  e \beta} \right)
\end{align*} 
for all $\beta>0$ which satisfy $2^{N+1} B {\bf P}  e \beta < 1$, for all $\epsilon>0$ and for two constants $A_1,A_2$ which are independent of $B$, $\beta$, $\epsilon$ and ${\bf n} $. The choice $\beta_0 \coloneqq \epsilon /( 2 A_2  (\sigma^2 + B^2 ) {\bf n} + \epsilon 2^{N+1} B {\bf P}  e ) $ approximately minimizes this last bound and we obtain the desired result if we use additionally that ${\bf P}  \le {\bf n} ^{N/(N+1)} \left(\prod_{i=1}^N \log n_i \right)$. 
\end{proof}

\appendix
\section{Appendix}
Davydov's inequality relates the covariance of two random variables to the $\alpha$-mixing coefficient:
\begin{proposition}[\cite{davydov1968convergence}]\label{davydovsIneq}
Let $\pspace$ be a probability space and let $\cG, \cH \subseteq \cA$ be sub-$\sigma$-algebras. Denote by $\alpha := \sup\{ |\p(A \cap B) - \p(A)\p(B) |:\, A \in \cG, B\in \cH \}$ the $\alpha$-mixing coefficient of $\cG$ and $\cH$. Let $p,q,r \ge 1$ be H{\"o}lder conjugate, i.e., $p^{-1} + q^{-1} + r^{-1} =1$. Let $\xi$ (resp. $\eta$) be in $L^p(\p)$ and $\cG$-measurable (resp. in $L^q(\p)$ and $\cH$-measurable). Then $\left| \text{Cov}(\xi,\eta) \right|  \le 12\, \alpha^{1/r} \norm{\xi}_{L^p(\p)} \norm{\eta}_{L^q(\p)}$.
\end{proposition}

\end{document}